\documentclass[12pt]{amsart}
\usepackage{amsmath}
\usepackage{amsthm}
\usepackage{amssymb}
\usepackage{amscd}
\usepackage{hyperref}
\usepackage{tikz}
\usetikzlibrary{matrix,arrows}

\DeclareMathOperator{\lc}{H}

\newcommand{\length}{\ell}

\newcommand{\eh}{\operatorname{e}}

\newcommand{\mf}{\mathfrak}

\DeclareMathOperator{\m}{\mathfrak{m}}

\newtheorem{theorem}{Theorem}
\newtheorem{lemma}[theorem]{Lemma}

\newtheorem{corollary}[theorem]{Corollary}

\newtheorem*{statement*}{Statement}
\newtheorem*{theorem*}{Theorem}
\newtheorem*{lemma*}{Lemma}
\newtheorem*{fact*}{Fact}

\theoremstyle{definition}
\newtheorem{definition}[theorem]{Definition}
\newtheorem*{definition*}{Definition}
\newtheorem{example}[theorem]{Example}
\newtheorem*{example*}{Example}

\theoremstyle{remark}
\newtheorem{remark}[theorem]{Remark}

\begin{document}

\title{Colength, multiplicity, and ideal closure operations}

\author{Linquan Ma}
\address{Department of Mathematics, Purdue University, West Lafayette, IN 47907 USA}
\email{ma326@purdue.edu}

\author{Pham Hung Quy}
\address{Department of Mathematics, FPT University, and Thang Long Institute of Mathematics and Applied Sciences, Hanoi, Vietnam}
\email{quyph@fe.edu.vn}

\author{Ilya Smirnov}
\address{Department of Mathematics, Stockholm University, S-10691, Stockholm, Sweden}
\email{smirnov@math.su.se}

\maketitle

\begin{center}
{\textit{Dedicated to Professor Bernd Ulrich on the occasion of his 65th birthday}}
\end{center}

\begin{abstract}
In a formally unmixed Noetherian local ring, if the colength and multiplicity of an integrally closed ideal agree, then $R$ is regular. We deduce this using the relationship between multiplicity and various ideal closure operations.
\end{abstract}

\section{Introduction}

Let $(R, \mf m)$ be a Noetherian local ring, $I$ be an $\mf m$-primary ideal, and $M$ be a finitely generated $R$-module of dimension $d$.
The Hilbert--Samuel multiplicity of $M$ with respect to $I$ is defined as
\[
\eh(I, M) = \lim_{n \to \infty} \frac{d! \length (M/I^nM)}{n^{d}}.
\]
We simplify our notation by letting $\eh(I):=\eh(I, R)$ and $\eh(R) := \eh(\mf m)$. The importance of the Hilbert--Samuel multiplicity in the study of singularities comes from Nagata's fundamental theorem: a Noetherian local ring $(R,\m)$ is regular if and only if it is formally unmixed and $\eh(R) = 1$. An ideal-theoretic concept naturally associated to multiplicity is integral closure.
Under mild assumptions on $R$, for a pair of ideals $J \subseteq I$ we have equality $\eh(I) = \eh(J)$
if and only if $I \subseteq \overline{J}$.

In this short note, we further the relationship
between multiplicity and integral closure by showing that in a formally equidimensional ring
$\eh(I) \geq \length (R/\overline{I})$ and
characterizing that in a formally unmixed ring the equality
holds for {\it some} parameter ideal if and only if $R$ is regular.
The latter is a vast generalization of Nagata's theorem:
we view his statement as $\eh(\mf m) = \length (R/\mf m)$. These results are obtained by investigating the relationship between multiplicity and various closure operations of parameter ideals. Let $J$ be an ideal generated by a system of parameters of $R$. We have the following containments of ideal closure operations under mild assumptions:
\[
J \subseteq J^{\lim} \subseteq J^* (\text{in characteristic } p > 0) \subseteq \overline{J}.
\]

The equalities between the multiplicity and the colength of these closures encode special properties of $R$ (again, under mild assumptions of $R$):
\begin{enumerate}
\item $\eh(J) = \length (R/J)$ for all (or some) $J$ if and only if $R$ is Cohen--Macaulay;
\item $\eh(J) = \length(R/J^{\lim})$ for all (or some) $J$ if and only if $R$ is Cohen--Macaulay (Le--Nguyen \cite{CuongNhan}, Theorem~\ref{is CM});
\item $\eh(J)= \length(R/J^*)$ for all (or some) $J$ if and only if $R$ is F-rational (Goto--Nakamura \cite{GotoNakamura}, Corollary~\ref{is Frat});
\item $\eh(J) = \length(R/\overline{J})$ for some $J$ if and only if $R$ is regular (Corollary~\ref{is regular}).
 \end{enumerate}

We remark that, our main contribution, Corollary~\ref{is regular}, also follows from the main result of \cite{WatanabeChain}, {\it if $(R,\m)$ is an excellent normal domain with an algebraically closed residue field}.\footnote{As pointed out in \cite[Lemma 2.1]{HunekeMaPhamSmirnov}, Watanabe's result in \cite{WatanabeChain} can be generalized to complete local domain with an algebraically closed residue field.} The point is that, under these assumptions of $R$, $\eh(I)=\length(R/I)$ for an integrally closed $\m$-primary ideal $I$ implies $\eh(\m)=1$ by \cite[Theorem 2.1]{WatanabeChain} (using Theorem \ref{inequality thm}), and hence $R$ is regular by Nagata's theorem. However, we do not see how to extend this approach to get the full version of Corollary \ref{is regular}.

\vspace{1em}

\noindent\textbf{Acknowledgement}: The authors thank Craig Huneke and Bernd Ulrich for valuable discussions, and Jugal Verma for comments on a draft of this note. The first author is supported in part by NSF Grant DMS $\#1901672$, and was supported by NSF Grant DMS $\#1836867/1600198$ when preparing this article. The second author is supported by Ministry of Education and Training, grant no. B2018-HHT-02. Part of this work has been done during a visit of the third author to Purdue University supported by Stiftelsen G S Magnusons fond of Kungliga Vetenskapsakademien. Finally, we thank the referee for her/his comments.

\section{Colength and multiplicity}

The goal of this section is to prove Theorem \ref{inequality thm}. This theorem can be also deduced from the methods in the next section. But we give an elementary approach here that avoids the use of limit closure and big Cohen-Macaulay algebras.

We recall that a Noetherian local ring $(R,\m)$ is equidimensional (resp., unmixed) if $\dim R/P = \dim R$ for every
minimal (resp., associated) prime $P$ of $R$. In other words, $R$ is
unmixed if it is equidimensional and $(S_1)$.
We say that a Noetherian local ring $R$ is formally equidimensional (resp., unmixed)
if $\widehat{R}$ is equidimensional (resp., unmixed). For an ideal $I\subseteq R$ and an element $x\in R$ we use $I:x^\infty$ to denote $\cup_n (I:x^n)$.

\begin{definition}
Let $x_1, \ldots, x_t$ be a sequence of elements in a Noetherian local ring $R$. We define $(x_1, \ldots, x_t)^\infty$ inductively as follows:
\begin{enumerate}
\item $(x_1)^\infty = (x_1) + 0:x_1^\infty$ if $t = 1$
\item $(x_1, \ldots, x_t)^\infty = (x_t)+ (x_1, \ldots, x_{t-1})^\infty:x_t^\infty$ if $t > 1$.
\end{enumerate}
\end{definition}

\begin{example}
The reader should be warned that this is not a closure operation on ideals, and
the result may depend on the order of elements. Consider $R= k[[x^4, x^3y, xy^3, y^4]]$.
Then $x^4, y^4$ form a system of parameters, but
\[
(x^4, y^4)^\infty = (x^4, x^6y^2, y^4) \neq (x^4, x^2y^6, y^4) = (y^4, x^4)^\infty.
\]
\end{example}

We record the following properties.

\begin{lemma}
Let $(R, \mf m)$ be a Noetherian local ring of dimension $d$.
For any sequence $x_1, \ldots, x_d$, $(x_1, \ldots, x_d)^\infty$ is either $\mf m$-primary or the unit ideal.
\end{lemma}
\begin{proof}
If $0:x_1^{\infty}$ is a proper ideal, i.e., $x_1 \notin \sqrt{(0)}$,
then $x_1$ is a regular element modulo $0:x_1^{\infty}$.
Hence $\dim R/(x_1)^\infty < d$ and we are done by induction.
\end{proof}

\begin{lemma}
\label{lem.mult is length}
Let $(R, \mf m)$ be a Noetherian local ring of dimension $d > 0$ and $x_1, \ldots, x_d$ be a system of parameters.
Then $\eh((x_1, \ldots, x_d)) = \length (R/(x_1, \ldots, x_d)^\infty)$.
\end{lemma}
\begin{proof}
Since $x_1$ is a parameter, it is not contained in any prime $\mf p$ of $\dim R/\mf p = d$, so $\dim (0:x_1^\infty) < d$.
Multiplicity is additive in short exact sequences, so $\eh((x_1, \ldots, x_d), R) = \eh((x_1, \ldots, x_d), R/0:x_1^\infty)$.
Because $x_1, \ldots, x_d$ is still a system of parameters on $R/0:x_1^\infty$
and $x_1$ is now a regular element, we have by \cite[Lemma~1]{Lech2}
\[
\eh((x_1, \ldots, x_d), R/0:x_1^\infty) = 
\eh((x_1, \ldots, x_d), R/(0:x_1^\infty, x_1)) = 
\eh((x_2, \ldots, x_d), R/(x_1)^\infty).
\]
The assertion now follows by induction on $d$. 
Note that for $d = 1$ the formula above gives that 
$\eh((x_1), R/0:x_1^\infty) = \length (R/(0:x_1^\infty, x_1)) = \length (R/(x_1)^\infty)$.
\end{proof}

\begin{remark}\label{extend rmk}
Let $(R, \mf m)$ be a Noetherian local ring and let $S = \widehat{R(t)} := \widehat{R[t]_{\mf m[t]}}$.
We note that $S$ is complete, has an infinite residue field,
and is a faithfully flat $R$-algebra such that $\mf m_RS$ is the maximal ideal of $S$.
It follows that $\length_R (R/I) = \length_S (S/IS)$ for every $\mf m$-primary ideal
and, thus, $\eh(I) = \eh(IS)$.
Moreover, if $I$ is integrally closed in $R$ then $IS$ is integrally closed in $S$. This
follows from \cite[Lemma~8.4.2 (9)]{SwansonHuneke}, which allows us to pass to $R(t)$,
and the fact that there is one-to-one correspondence between $\mf m$-primary ideals
in $R$ and $\widehat{R}$, so if $I\widehat{R}$ is a reduction of a larger ideal,
then $I$ is a reduction too.
\end{remark}

\begin{theorem}\label{inequality thm}
Let $(R, \mf m)$ be a formally equidimensional Noetherian local ring.
Then for every $\m$-primary integrally closed ideal $I$ we have $\eh(I) \geq \length (R/I)$.
\end{theorem}
\begin{proof}
We may pass from $R$ to $R(t)$ without changing the colength and the integral closedness of $I$. Thus we assume that $R$ has an infinite residue field.
Let $(x_1, \ldots, x_d)$ be a minimal reduction of $I$. By Lemma \ref{lem.mult is length}, it is enough to show that
$(x_1, \ldots, x_d)^\infty \subseteq I$. This is a consequence of colon-capturing
(\cite{Ratliff}, \cite[Theorem~5.4.1]{SwansonHuneke}).
Namely, it is clear that $(x_1)^\infty = (x_1)+0:x_1^\infty \subseteq \overline{(x_1)}$,
and for $i > 1$  we can use induction to see that
\[
(x_1, \ldots, x_{i})^\infty
= (x_i)+ (x_1, \ldots, x_{i - 1})^\infty:x_i^\infty
\subseteq (x_i)+ \overline{(x_1, \ldots, x_{i - 1})}:x_i^\infty \subseteq \overline{(x_1, \ldots, x_{i})}. \qedhere
\]
\end{proof}

\begin{example}
The equidimensionality assumption in Theorem \ref{inequality thm} is necessary. Let $R=k[[x,y,z]]/(xy, xz)$ and consider the ideal $(x^n, y,z)$.
One can check that this ideal is integrally closed, has multiplicity $1$, and colength $n$.
\end{example}

\section{Limit closure, integral closure, and the main result}
In this section we study a relation between multiplicity and the colength of limit closure, and we prove our main result. As a byproduct of our methods, we also recover some results in \cite{CuongNhan} and \cite{GotoNakamura}.

\begin{definition}
Let $(R,\m)$ be a Noetherian local ring and let $x_1, \ldots, x_d$ be a system of parameters of $R$. The limit closure of $(x_1, \ldots, x_d)$ in $R$ is defined as
\[
(x_1, \ldots, x_d)^{\lim_R} = \bigcup_{n\geq 0} (x_1^{n+1}, \ldots, x_d^{n + 1}):_R (x_1\cdots x_d)^n.
\]
We will write $(x_1, \ldots, x_d)^{\lim}$ if $R$ is clear from the context.
\end{definition}

We note that $(x_1, \ldots, x_d)^{\lim}/(x_1, \ldots, x_d)$ is the kernel of the natural map $R/(x_1, \ldots, x_d) \to \lc_{\mf m}^d (R)$: since $\lc_{\mf m}^d (R)=\varinjlim_n\frac{R}{(x_1^n,\dots,x_d^n)}$ with connection map multiplication by $x_1\cdots x_d$, $\overline{z}\in R/(x_1, \ldots, x_d)$ maps to $0$ in $\lc_{\mf m}^d (R)$ if and only if $z(x_1\cdots x_d)^n\in (x_1^{n+1},\dots,x_d^{n+1})$ for some $n$, that is, $z\in (x_1, \ldots, x_d)^{\lim}$.
In particular, limit closure of an ideal generated by a system of parameters is independent of the choice of the generators. In general, limit closure is hard to study: Hochster's monomial conjecture/theorem simply says that $(x_1,\dots,x_d)^{\lim}$ is not the unit ideal. This was proved by Hochster in the equal characteristic case \cite{HochsterMonomial} and was proved by Andr\'{e} in mixed characteristic \cite{Andre}.

The next theorem is a crucial ingredient towards proving our main result. It follows from \cite[Theorem~3.1]{CuongNhan}. But we provide a different and simpler proof.

\begin{theorem}\label{is CM}
Let $(R, \mf m)$ be a Noetherian local ring of dimension $d$. Then for every system of parameters $x_1, \ldots, x_d$, we have
\[
\eh((x_1, \ldots, x_d)) \geq \length (R/(x_1, \ldots, x_d)^{\lim}).
\]
Moreover, if $R$ is unmixed and is a homomorphic image of a Cohen--Macaulay ring,
then the equality holds for one (equivalently, all) system of parameters if and only if $R$ is Cohen--Macaulay.
\end{theorem}
\begin{proof}
The first assertion is well-known (for example, see \cite[Lemma~2.3]{NNN}). The point is that, by Lech's formula \cite{Lech}, $\eh((x_1,\dots,x_d))=\lim_{n\to\infty}\frac{\length(R/(x_1^n,\dots,x_d^n))}{n^d}$. We can filter $R/(x_1^n,\dots,x_d^n)$ by $n^d$ ideals generated by monomials in $x_1,\dots,x_d$, and it is easy to check that each factor maps onto $R/(x_1,\dots,x_d)^{\lim}$.\footnote{For example, if $d=1$, the we have a filtration $(x_1^n)\subseteq(x_1^{n-1})\subseteq\cdots\subseteq (x_1)\subseteq R$, the $i$-th factor ${(x_1^i)}/{(x_1^{i+1})}\cong R/(x_1^{i+1}:x_1^i)$, since $(x_1^{i+1}:x_1^i)\subseteq (x_1)^{\lim}$ by definition, ${(x_1^i)}/{(x_1^{i+1})} \twoheadrightarrow R/(x_1)^{\lim}$. In the general case, each factor looks like $({J, x_1^{n_1}\cdots x_d^{n_d}})/{J}\cong R/(J:x_1^{n_1}\cdots x_d^{n_d})$ where $J$ is an $\m$-primary ideal generated by monomials $x_1^{j_1}\cdots x_d^{j_d}$ in $x_1,\dots,x_d$ such that $j_i>n_i$ for some $i$, i.e., at least one exponent is bigger than that appearing in $x_1^{n_1}\cdots x_d^{n_d}$. Now for every $y\in J:(x_1^{n_1}\cdots x_d^{n_d})$, we have $yx_1^{n_1}\cdots x_d^{n_d}=\sum a_{j_1\dots j_d}x_1^{j_1}\cdots x_d^{j_d}$. Pick $n$ that is larger than all $n_i$ and multiply this equation by $x_1^{n-n_1}\cdots x_d^{n-n_d}$ we get $y(x_1\cdots x_d)^n=\sum a_{j_1\dots j_d}x_1^{j_1+n-n_1}\cdots x_d^{j_d+n-n_d}\in (x_1^{n+1},\dots,x_d^{n+1})$ by the assumptions on $j_i$. Hence $y\in (x_1,\dots,x_d)^{\lim}$ and thus $R/(J:x_1^{n_1}\cdots x_d^{n_d})\twoheadrightarrow R/(x_1,\dots,x_d)^{\lim}$.}

Now we prove the second assertion. We may assume the residue field of $R$ is infinite. We proceed by induction on $d$. If $d=1$ the assertion is obvious. If $d=2$, the statement follows from \cite[Theorem~1.5]{limitclosure}.\footnote{Note that the ``unmixed" assumption in \cite[Theorem~1.5]{limitclosure} means formally unmixed in our context, and if $R$ is a homomorphic image of a Cohen--Macaulay ring, then $R$ is unmixed implies $R$ is formally unmixed \cite[Theorem 2.1.15]{BrunsHerzog}.} Now we assume $d\geq 3$, it follows from \cite[Proposition~4.16]{CuongQuy} that
if $z \in (x_1,\dots,x_d)$ is general, then $R':=R/zR$ is equidimensional and $(S_1)$ on the punctured spectrum.
Let $S = R'/\lc^0_{\mf m}(R')$. We know that $S$ is unmixed.
Since $\lc^0_{\mf m}(R')$ has finite length and $z$ is a general element in $(x_1,\dots,x_d)$, we have
$$\eh((x_1, \ldots, x_d), S) = \eh((x_1, \ldots, x_d), R') = \eh((x_1, \ldots, x_d)).$$
Replacing $x_1,\dots, x_{d-1}$ if necessary, we may assume that $x_1, \ldots, x_{d-1}, z$ form a system of parameters of $R$, and thus $x_1,\dots,x_{d-1}$ form a system of parameters on $R'$ and $S$.
By \cite[Theorem~1.2 and Proposition~2.7]{limitclosure} , we know that
$(x_1, \ldots, x_{d-1})^{\lim_{R'}}S = (x_1, \ldots, x_{d-1})^{\lim_{S}}$.
Moreover,
if $r\in (x_1^{n+1}, \ldots, x_{d - 1}^{n+1}, z) : (x_1\cdots x_{d-1})^n$,
then
\[r  (x_1\cdots x_{d-1}z)^n \subseteq z^n (x_1^{n+1}, \ldots, x_{d - 1}^{n+1}, z)
\subseteq (x_1^{n+1}, \ldots, x_{d - 1}^{n+1}, z^{n+1}).
\]
This implies that the pre-image of $(x_1,\dots,x_{d-1})^{\lim_{R'}}$ in $R$ is contained in $(x_1,\dots,x_d)^{\lim_R}$. Thus we have
\begin{align*}
\eh((x_1, \ldots, x_d), S) &\geq \length (S/(x_1, \ldots, x_{d-1})^{\lim_S})
= \length (R'/(x_1, \ldots, x_{d-1})^{\lim_{R'}})
\\ &\geq \length (R/(x_1, \ldots, x_d)^{\lim_R}) = \eh((x_1, \ldots, x_d)) = \eh((x_1, \ldots, x_d), S)
\end{align*}
and so we must have equalities all over.
Therefore $S$ is Cohen--Macaulay by the induction hypothesis, and
it follows that $\lc^i_{\mf m} (R') \cong \lc^i_{\mf m}(S) = 0$ for $0 < i < \dim S=\dim R'$.

Finally, since $R$ is unmixed, $z$ is a regular element, so the sequence
\[
0 \to R \xrightarrow{\times z} R \to R'=R/zR \to 0
\]
is exact and induces the exact sequence
\[
0 \to \lc^0_{\mf m} (R/zR) \to \lc^1_{\mf m} (R) \xrightarrow{\times z} \lc^1_{\mf m} (R) \to 0.
\]
Because $R$ is unmixed, $\lc^1_{\mf m} (R)$ has finite length. The sequence above then implies that
$\lc^0_{\mf m} (R/zR) = 0$. Thus $R/zR$ is Cohen--Macaulay, so $R$ is Cohen--Macaulay.
\end{proof}

Using limit closure we recover the main result of \cite[Theorem 1.2]{GotoNakamura}, see also \cite[Corollary 1.9 and Remark 1.10]{CiupercaEnescu}.

\begin{corollary}\label{is Frat}
Let $(R, \mf m)$ be an equidimensional Noetherian local ring of characteristic $p > 0$
which is a homomorphic image of a Cohen--Macaulay ring.
Then for any system of parameters $x_1, \ldots, x_d$ of $R$ we have
$\eh((x_1, \ldots, x_d)) \geq \length (R/(x_1, \ldots, x_d)^*)$.
Moreover, if, in addition, $R$ is unmixed, then
the equality holds for one (equivalently, all) system of parameters if and only if $R$ is F-rational.
\end{corollary}
\begin{proof}
The first assertion follows from Theorem \ref{is CM} and colon-capturing: $(x_1, \ldots, x_d)^{\lim} \subseteq (x_1, \ldots, x_d)^*$, see \cite[Theorem 2.3 and Remark~5.4]{Huneke}.
If $R$ is unmixed and equality holds, then by Theorem~\ref{is CM}, $R$ is Cohen--Macaulay and thus $\eh((x_1, \ldots, x_d))=\length(R/(x_1,\dots,x_d))$. Hence $(x_1, \ldots, x_d)=(x_1,\dots,x_d)^*$, so $R$ is F-rational by \cite[Proposition~2.2]{FedderWatanabe}.
\end{proof}

We next show that limit closure is contained in the integral closure in all characteristics using the existence of big Cohen--Macaulay algebras.

\begin{theorem}\label{homo limit}
Let $(R, \mf m)$ be a formally equidimensional Noetherian local ring, then for every system of parameters $x_1,\ldots,x_d$ we have
\[
(x_1, \ldots, x_d)^{\lim} \subseteq \overline{(x_1, \ldots, x_d)}.
\]
\end{theorem}
\begin{proof}
We may assume that $R$ is complete.
To check whether an element is in the integral closure, it is enough to check this modulo every minimal prime of $R$.
Since $R$ is equidimensional, $x_1, \ldots,x_d$ is still a system of parameters modulo every minimal prime of $R$.
So if $r$ is in $(x_1, \ldots, x_d)^{\lim}$, then this is also true modulo every minimal prime of $R$.
Therefore we reduce to the case that $R$ is a complete local domain.

Now let $B$ be a big Cohen--Macaulay $R$-algebra, whose existence follows from \cite{HochsterHunekeBig} and \cite{HochsterHunekeApplicationsofBigCM} in equal characteristic, and from \cite{Andre} (see also \cite{HeitmannMa}) in mixed characteristic. If
$r \in (x_1, \ldots, x_d)^{\lim}$, then
$r\in (x_1^t, \ldots, x_d^t):_R(x_1\cdots x_d)^{t-1}$ for some $t$. It follows that
\[r \in \left((x_1^t, \ldots, x_d^t):_B(x_1\cdots x_d)^{t-1} \right)\cap R= (x_1,\ldots ,x_d)B\cap R,\]
since $x_1,\ldots,x_d$ is a regular sequence on $B$.

Thus it is enough to prove that  $(x_1,\ldots,x_d)B\cap R$ is contained in $\overline{(x_1,\ldots,x_d)}$.
In fact, $JB\cap R$ is contained in $\overline{J}$ for every ideal $J$ of $R$: since $R$ is a complete local domain and $B$ is a big Cohen--Macaulay algebra, $B$ is a solid $R$-algebra in the sense of \cite[Corollary~10.6]{Hochster}, thus $JB\cap R$ is contained in the solid closure of $J$, but solid closure is always contained in the integral closure by \cite[Theorem~5.10]{Hochster}.
\end{proof}

\begin{corollary}\label{is regular}
Let $(R, \mf m)$ be a Noetherian local ring that is formally equidimensional. Then for every $\m$-primary integrally closed ideal $I$, we have $\eh(I) \geq \length (R/I)$. Moreover, if, in addition, $R$ is formally unmixed and equality holds for some $I$, then $R$ is regular.
\end{corollary}
\begin{proof}
We may assume that $R$ is complete with an infinite residue field by Remark \ref{extend rmk}.
Let $(x_1, \ldots, x_d)$ be a minimal reduction of $I$. By Theorem \ref{is CM} and Theorem~\ref{homo limit},
$$\eh(I)=\eh ((x_1, \ldots, x_d)) \geq \length (R/(x_1, \ldots, x_d)^{\lim})\geq \length(R/I).$$
Now if $R$ is formally unmixed and $\eh(I)=\length(R/I)$, then $\eh ((x_1, \ldots, x_d)) =\length (R/(x_1, \ldots, x_d)^{\lim})$ so by Theorem~\ref{is CM}, $R$ is Cohen--Macaulay. But then $\eh ((x_1, \ldots, x_d))=\length(R/(x_1,\dots,x_d))$ and hence $I = (x_1, \ldots, x_d)$, so $R$ is regular by \cite[Corollary~2.5]{Goto}.
\end{proof}

We would like to note that \cite[Theorem~1.1]{Goto} shows that an integrally closed $\mf m$-primary parameter ideal
in a regular ring $(R, \mf m)$ has the form $x_1^n, x_2, \ldots, x_{d}$ where $x_1, \ldots, x_d$ are minimal generators
of $\mf m$.

\begin{example}
One might ask that, in a formally unmixed Noetherian local ring $(R,\m)$, whether $(x_1,\dots,x_d)^{\lim} = \overline{(x_1,\dots,x_d)}$ for a system of parameters already implies $R$ is regular. However this is not true in general: Let $R = k[[a,b,c,d]]/(a,b)\cap (c,d)$.
Then $R$ is complete, unmixed, has dimension $2$, with $\eh(R) = 2$ and $\lc^1_\mf m (R) \cong k$.
Let $(x,y)$ be a minimal reduction of $\mf m$. It follows from \cite[Theorem~1.5]{limitclosure}
that $\length (R/(x,y)^{\lim}) = 1$, so $(x,y)^{\lim} =\m= \overline{(x,y)}$.
\end{example}

\bibliographystyle{plain}
\bibliography{refs}

\end{document}